\documentclass[11pt]{amsart}

\usepackage[T1]{fontenc}
\usepackage[utf8]{inputenc}
\usepackage{lmodern}
\usepackage{microtype}
\usepackage{amsmath,amssymb,amsthm,mathtools}
\usepackage{enumitem}
\usepackage[colorlinks=true,linkcolor=blue,citecolor=blue,urlcolor=blue]{hyperref}

\makeatletter
\@ifundefined{subjclassname@2020}{%
  \@namedef{subjclassname@2020}{\textup{2020} Mathematics Subject Classification}%
}{}
\makeatother

\newtheorem{thm}{Theorem}[section]
\newtheorem{lem}[thm]{Lemma}
\newtheorem{prop}[thm]{Proposition}
\newtheorem{cor}[thm]{Corollary}
\theoremstyle{remark}
\newtheorem{rmk}{Remark}

\newcommand{\PP}{\mathbb{P}}

\title{On the problem of large gcd for disjoint residue classes}
\author{Jan Fornal}
	\address{
		Department of Mathematics,
		University of Bristol \\
		Beacon House, Queens Rd, Bristol BS8 1QU}
	\email{nc24166@bristol.ac.uk}
\author{Yu-Chen Sun}
\address{Department of Mathematics,
		University of Bristol \\
		Beacon House, Queens Rd, Bristol BS8 1QU}
\email{yuchensun93@163.com}
\date{}

\begin{document}

\begin{abstract}
  Consider $k$ pairwise disjoint residue classes $a_i \pmod{m_i}$. We prove
  that
  \[
    \max_{1\leq i<j\leq k}\gcd(m_i,m_j)
    \gg k\exp\!\left(-(2+o(1))
      \sqrt{\frac{\log k}{\log\log k}}\right).
  \]
  The proof uses a complete graph whose edges are colored by the gcds of the
  corresponding moduli, together with a structural lemma, a sieve-theoretic
  partition, M\"obius inversion, and the discrete Fourier transform.
\end{abstract}

\maketitle

\section{Introduction}

The greatest common divisor appears in many problems in number theory.  One
example is the study of GCD sums
\[
  \sum_{r,s\leq N}
  \frac{\gcd(n_r,n_s)^{2\alpha}}{(n_rn_s)^\alpha}.
\]
Aistleitner, Berkes and Seip obtained estimates for these sums and applied
them to systems of dilated functions and metric number
theory~\cite{aistleitner-berkes-seip2015}.  Bondarenko and Seip used large GCD
sums in their work on extreme values of the Riemann zeta
function~\cite{bondarenko-seip2017}.  GCD graphs also appear in the proof of
the Duffin--Schaeffer conjecture by Koukoulopoulos and
Maynard~\cite{koukoulopoulos-maynard2020}.  In that proof, bipartite GCD graphs
are used to keep track of pairs of denominators with common prime factors.

Green and Walker proved the following related result~\cite{green-walker2021}.
Let $A\subseteq[X,2X]$ and $B\subseteq[Y,2Y]$.  Suppose that
$\gcd(a,b)\geq D$ for at least $\delta|A||B|$ pairs
$(a,b)\in A\times B$.  Then, for every $\varepsilon>0$ and
$D\leq\min\{X,Y\}$,
\[
  |A||B|\ll_\varepsilon
  \delta^{-2-\varepsilon}\frac{XY}{D^2}.
\]
Thus, if many pairs have gcd at least $D$, then the product $|A||B|$ is
bounded in terms of $X$, $Y$, $D$ and $\delta$.  In this paper we consider a
related question for pairwise disjoint residue classes.  We ask how large
$\gcd(m_i,m_j)$ must be for at least one pair of moduli.

For two residue classes, the relation with the gcd follows from the Chinese remainder theorem.  Namely,
\[
  a_i\pmod{m_i}\quad\text{and}\quad a_j\pmod{m_j}
\]
intersect if and only if
\[
  a_i\equiv a_j\pmod{\gcd(m_i,m_j)}.
\]
In particular, two residue classes with coprime moduli always intersect.
Therefore, in a pairwise disjoint family, one has
$\gcd(m_i,m_j)>1$ for every $i<j$.  This observation does not give a lower
bound depending on $k$, since different pairs may have different common
divisors.

Sun conjectured the following lower bound~\cite{sun2006}:
\begin{equation}\label{eq:sun-conjecture}
  \max_{1\leq i<j\leq k}\gcd(m_i,m_j)\geq k
\end{equation}
for every family of pairwise disjoint residue classes
$a_1\pmod{m_1},\ldots,a_k\pmod{m_k}$.  This bound would be best possible,
since the $k$ distinct residue classes modulo $k$ are pairwise disjoint.

Sun also gave a group-theoretic version of this conjecture in his work on
finite covers of groups~\cite{sun2006}.  Let
$a_1G_1,\ldots,a_kG_k$ be pairwise disjoint left cosets in a group $G$, where
each $G_i$ has finite index.  The conjecture states that
\[
  \gcd\bigl([G:G_i],[G:G_j]\bigr)\geq k
\]
for some $i<j$.  The choice $G=\mathbb Z$ and $G_i=m_i\mathbb Z$ gives
\eqref{eq:sun-conjecture}.  O'Bryant proved the integer conjecture for
$k\leq20$~\cite{obryant2007}.  Zhu proved the group-theoretic conjecture for
$k=3,4$~\cite{zhu2008}.  Sun proved the cases $k=2$ and the case in which $G$
is a finite $p$-group~\cite{sun2006}.

In this paper we prove the following estimate.

\begin{thm}\label{thm:gcd-lower-bound}
  Let
  \[
    a_1 \pmod{m_1}, \ldots, a_k \pmod{m_k}
  \]
  be pairwise disjoint residue classes. Then
  \[
    \max_{1 \leq i < j \leq k} \gcd(m_i, m_j)
    \gg k \exp\!\left(-(2 + o(1))
      \sqrt{\frac{\log k}{\log \log k}}\right).
  \]
\end{thm}

In particular, Theorem~\ref{thm:gcd-lower-bound} implies that
\[
  \max_{i<j}\gcd(m_i,m_j)\geq k^{1-o(1)}.
\]
Thus, we obtain the bound in Sun's conjecture up to a factor $k^{o(1)}$.

We briefly describe the main idea of the proof.  Let
\[
  d=\max_{1\leq i<j\leq k}\gcd(m_i,m_j).
\]
We regard the residue classes as the vertices of a complete graph and color
the edge between two vertices by the gcd of their moduli.  For each
$n\leq d$, we introduce a set $K_n$ of vertices whose moduli are divisible by
$n$, but not by a proper multiple of $n$ which is at most $d$.  A vertex may
belong to several such sets, so we assign weights in such a way that its total
weight over all the sets $K_n$ is equal to $1$.

The key combinatorial input is
Lemma~\ref{lem:relative-structural-decomposition}.  It gives a useful
dichotomy for the edges between $K_m$ and $K_n$ whose colors are different
from $\gcd(m,n)$.  Either one endpoint belongs to a small exceptional set, or,
after fixing a vertex outside the exceptional sets, the assigned weight of
that vertex is small whenever it is incident to many such edges.  Thus the
contribution of these edges is controlled either by the small size of an
exceptional set or by the small weight of a fixed endpoint.

We then partition $[d]$ into a small number of sets $C_i$.  On each $C_i$, a
divisor-sum parameter $T_i$ controls the corresponding weighted gcd sum.  If $k_i$
denotes the total weight carried by the sets $K_n$ with $n\in C_i$, M\"obius
inversion and Fourier orthogonality give a quadratic lower bound $k_i^2$.
The generalized Chinese remainder theorem eliminates the terms whose edge
color is $\gcd(m,n)$, while
Lemma~\ref{lem:relative-structural-decomposition} controls the remaining
terms.  Together they give the matching upper bound
\[
  k_i^2\ll k_i dT_i\log d.
\]
Consequently, $k_i\ll dT_i\log d$.  Summing over the parts of the partition
and using the estimates for their number and for $T_i$, we obtain
\[
  k\ll d\exp\!\left((2+o(1))
    \sqrt{\frac{\log d}{\log\log d}}\right).
\]
Solving this inequality for $d$ gives Theorem~\ref{thm:gcd-lower-bound}.

There is another related problem in which the moduli are assumed to be
distinct and bounded.  Let $f(x)$ be the largest size of a family of pairwise
disjoint arithmetic progressions with moduli satisfying
\[
  2\leq m_1<\cdots<m_k\leq x.
\]
Erd\H{o}s and Stein conjectured that $f(x)=o(x)$.  Erd\H{o}s and
Szemer\'edi proved this conjecture and obtained the first quantitative
estimates~\cite{erdos-szemeredi1968}.  Let
\[
  L(x):=\exp\!\left(\sqrt{\log x\log\log x}\right).
\]
Croot proved that
\[
  xL(x)^{-\sqrt2+o(1)}\leq f(x)
  \leq xL(x)^{-1/6+o(1)},
\]
and proved the upper bound with exponent $1/2$ when the moduli are
square-free~\cite{croot2003}.  Chen proved the same upper bound without the
square-free condition~\cite{chen2005}.  De la Bret\`eche, Ford and Vandehey
later proved
\begin{equation}
  xL(x)^{-1+o(1)}\leq f(x)
  \leq xL(x)^{-\sqrt3/2+o(1)},
\end{equation}
and conjectured the formula~\cite{bdfv2013}
\[
  f(x)=xL(x)^{-1+o(1)}.
\]
A recent preprint of Ho gives a proof of this conjecture~\cite{ho2026}.
Bloom's Erd\H{o}s Problems website now records Problem 202 as
solved~\cite{bloom-erdos202}.

The problem defining $f(x)$ assumes that the moduli are distinct and bounded,
while Theorem~\ref{thm:gcd-lower-bound} has neither assumption.  Combining
the formula for $f(x)$ with Theorem~\ref{thm:gcd-lower-bound} gives the
following consequence.

\begin{cor}
  Let $\mathcal Q_x\subseteq[1,x]\cap\mathbb N$ be a set of distinct moduli
  for which there exist pairwise disjoint residue classes
  $a_q\pmod q$, $q\in\mathcal Q_x$.  Suppose that, as $x\to\infty$,
  \[
    |\mathcal Q_x|=xL(x)^{-1+o(1)}.
  \]
  Then
  \[
    \max_{\substack{q,q'\in\mathcal Q_x\\q\neq q'}}\gcd(q,q')
    \geq xL(x)^{-1+o(1)}.
  \]
  Moreover, there exist distinct $q,q'\in\mathcal Q_x$ and positive integers
  $g,u,v$ such that
  \[
    q=gu,\qquad q'=gv,\qquad \gcd(u,v)=1,
  \]
  and
  \[
    g\geq xL(x)^{-1+o(1)},
    \qquad \max\{u,v\}\leq L(x)^{1+o(1)}.
  \]
  In particular, this holds for every extremal family with
  $|\mathcal Q_x|=f(x)$.
\end{cor}

\begin{proof}
  Let $k=|\mathcal Q_x|$.  By Theorem~\ref{thm:gcd-lower-bound},
  \[
    \max_{\substack{q,q'\in\mathcal Q_x\\q\neq q'}}\gcd(q,q')
    \gg k\exp\!\left(-(2+o(1))
      \sqrt{\frac{\log k}{\log\log k}}\right).
  \]
  Since $\log k=(1+o(1))\log x$, we have
  \[
    \sqrt{\frac{\log k}{\log\log k}}
    =o\!\left(\sqrt{\log x\log\log x}\right).
  \]
  The required bound now follows from
  $k=xL(x)^{-1+o(1)}$.
  Choose distinct $q,q'\in\mathcal Q_x$ satisfying this bound and let
  $g=\gcd(q,q')$, $u=q/g$ and $v=q'/g$.  Then $\gcd(u,v)=1$, while
  \[
    \max\{u,v\}\leq \frac{x}{g}\leq L(x)^{1+o(1)}.
  \]
\end{proof}

Thus every extremal family contains two moduli with a common multiplicative
core of size $xL(x)^{-1+o(1)}$ and comparatively small coprime tails.

The rest of the paper is organized as follows.  In
Section~\ref{sec:reduction}, we introduce the colored graph and its weights,
state the two main propositions, and deduce
Theorem~\ref{thm:gcd-lower-bound}.  In
Section~\ref{sec:sieve-theory}, we prove the sieve-theoretic partition
proposition.  Section~\ref{sec:structural-lemma} is devoted to the key
structural lemma for the colored graph, and in
Section~\ref{sec:weighted-estimate} we prove the
weighted estimate for $k_i$.

\section*{Acknowledgements}

We thank Yuchen Ding and Xiamiao Zhao for helpful comments and discussions.

\section{Proof of Theorem~\ref{thm:gcd-lower-bound}}\label{sec:reduction}

We can naturally rephrase the statement of the theorem. Suppose that:
\begin{equation} \label{eq:max-gcd-d}
  \max_{1 \leq i < j \leq k} \gcd(m_i, m_j) = d
\end{equation}
Then Theorem~\ref{thm:gcd-lower-bound} follows once we prove
\[
  k \ll d \exp\!\left((2 + o(1))
  \sqrt{\frac{\log d}{\log \log d}}\right).
\]

We begin with the graph and weight construction that decodes the structure of a family of residue classes satisfying \eqref{eq:max-gcd-d}. As usual, we will work on colored graph $(V, E)$, where $V = \{(a_s, m_s) : s \in [k] \}$ and every two vertices $v_1, v_2 \in V$ are adjacent with an edge of color $\gcd(m(v_1), m(v_2))$. For each $j \in [d]$, let $L_j$ be the set of vertices $v$ such that $j | m(v)$. For example, it is clear that $L_1 = V$. Let us also assign to $K_j$ for every $j \in [d]$, the set of vertices that is equal to:
\begin{equation}
  L_j \backslash \Bigl( \bigcup_{s = 2}^{\lfloor d/j \rfloor} L_{sj} \Bigr)
\end{equation}
Observe that:
\begin{equation} \label{eq:union-k-j}
  V = \bigcup_{j = 1}^d K_j
\end{equation}
Indeed, for each $v \in V$, let $w \in [d]$ be the largest number so that $v \in L_w$ (for sure, $v \in L_1$ so there must be such a maximal $w$). This means that:
\begin{equation} 
    v \in L_w \backslash \Bigl( \bigcup_{s = 2}^{\lfloor d/w \rfloor} L_{sw} \Bigr) = K_w
\end{equation}
One of the instant conclusions from \eqref{eq:union-k-j} is that 
\begin{equation}
  k \leq \sum_{j = 1}^d |K_j|
\end{equation}
For each vertex $v$, let us introduce the weight $w(v) = \frac{1}{\# \{j \in [d] : v \in K_j \} }$, the simple double counting argument implies that:
\begin{equation}
  k = \sum_{j = 1}^d \sum_{v \in K_j} w(v)
\end{equation}
We will also use the notation: $w(v, n) = \mathbf{1}_{v \in K_n} w(v)$.

We first record the sieve-theoretic input.
\begin{prop} \label{prop:sieve-theoretic-partition}
  For every sufficiently large integer $d$, there exists a partition:
  \begin{equation}
    [d] = \bigsqcup_{i \in \mathcal{I}} C_i
  \end{equation}
  such that the following two conditions are satisfied:
  \begin{enumerate}
    \item The size of the set $|\mathcal{I}| < \exp(o(\sqrt{\frac{\log d}{\log \log d}}))$.
    \item For each $i \in \mathcal{I}$, one has that:
    \begin{equation}\label{eq:Ti-definition}
      T_i := \frac{1}{d} \max_{m \in C_i} \sum_{e | m} \sum_{\substack{n \in C_i \\ \gcd(n, m) = e}} e
    \end{equation}
    is bounded by $\exp((2 + o(1)) \sqrt{\frac{\log d}{\log \log d}})$.
  \end{enumerate}
\end{prop}

\begin{prop}\label{prop:weighted-partition-bound}
  With the notation above, let
  $[d]=\bigsqcup_{i\in\mathcal I}C_i$ be any partition, and for each
  $i\in\mathcal I$, let $T_i$ be as in \eqref{eq:Ti-definition}, and define
  \begin{equation}\label{eq:ki-definition}
    k_i:=\sum_{n\in C_i}\sum_{v\in V}w(v,n)
  \end{equation}
  Then, for every $i\in\mathcal I$,
  \[
    k_i\ll dT_i\log d.
  \]
\end{prop}

\begin{proof}[Proof of Theorem~\ref{thm:gcd-lower-bound}, assuming
  Propositions~\ref{prop:sieve-theoretic-partition}
  and~\ref{prop:weighted-partition-bound}]
  Let $d$ be as in \eqref{eq:max-gcd-d}. If $d>k$, the theorem is immediate. We may therefore assume $d\leq k$.
  Let $[d]=\bigsqcup_{i\in\mathcal I}C_i$ be the partition supplied by Proposition \ref{prop:sieve-theoretic-partition}. For each $i\in\mathcal I$, let $k_i$ be as in \eqref{eq:ki-definition}. Since $k=\sum_{i\in\mathcal I}k_i$, Proposition \ref{prop:weighted-partition-bound} gives
  \[
    k
    \ll d\log d \sum_{i\in\mathcal I}T_i
    \leq d\log d\,|\mathcal I|\max_{i\in\mathcal I}T_i.
  \]
  By the two conclusions of Proposition \ref{prop:sieve-theoretic-partition},
  \[
    |\mathcal I|\leq \exp\!\left(o\!\left(\sqrt{\frac{\log d}{\log\log d}}\right)\right),
    \qquad
    \max_i T_i\leq \exp\!\left((2+o(1))\sqrt{\frac{\log d}{\log\log d}}\right).
  \]
  Since $\log d=\exp(o(\sqrt{\log d/\log\log d}))$, this gives
  \[
    k\ll d\exp\!\left((2+o(1))\sqrt{\frac{\log d}{\log\log d}}\right).
  \]
  By the eventual monotonicity of $\sqrt{\log x/\log\log x}$, the last bound implies
  \[
    d\gg k\exp\!\left(-(2+o(1))\sqrt{\frac{\log k}{\log\log k}}\right),
  \]
  which is the desired estimate.
\end{proof}

\section{A sieve-theoretic partition}
\label{sec:sieve-theory}

We use a sieve-theoretic idea to construct the partition in
Proposition~\ref{prop:sieve-theoretic-partition}. We separate the very small
prime factors of each integer and divide the remaining primes into short
intervals. We then group the integers according to their small-prime parts and
the numbers of prime factors, counted with multiplicity, in these intervals.
This multiscale decomposition produces only a small number of classes, while
the resulting uniformity within each class allows us to control $T_i$.

Assume that $d$ is sufficiently large. Let:
\[
 \eta=(\log \log d)^{-1/2},
\]
letting $J \ll (\log \log d)^{3/2}$, and define
\[
 U_0=\lfloor \log \log \log d \rfloor \quad \text{and} \quad U_{j+1}=(1+\eta)U_j\qquad(0 \leq j \leq J).
\]
For $0\le j<J$, let
\begin{equation}
 \mathcal P_j=
 \{p\text{ prime}:\mathrm e^{U_j}<p\le
 e^{U_{j+1}}\},
\end{equation}
Let $\Omega_j(n)$ denote the total number of prime factors of $n$ in $\mathcal P_j$. For given $n$, define its very small prime part by
\begin{equation}
 a(n)=\prod_{p\le \mathrm e^{U_0}}p^{v_p(n)},
\end{equation}
For each prime $p \leq e^{U_0}$ there are at most $\log d$ possible choices for $v_p(n)$ , therefore $a(n)$ belongs to the set of size $(\log d)^{\log \log d} = \exp((\log \log d)^2)$. For a possible value $a$ and a vector
$\mathbf r=(r_0,\ldots,r_{J-1})$, let
\begin{equation}\label{eq:multiscale-class}
 i = (a,\mathbf r)=
 \{n\le d:a(n)=a,\ \Omega_j(n)=r_j\text{ for every }j\}.
\end{equation}

\begin{lem}
The number of nonempty classes in \eqref{eq:multiscale-class} is bound by
\[
\exp((\log \log d)^2) \exp((\log \log d)^{5/2}) \ll \exp((1 + o(1)) (\log \log d)^{5/2}).
\]
\end{lem}

\begin{proof}
Obviously $r_j \in [0, \log d]$ and $\sum_{1 \leq j \leq J}r_j = \Omega(n)-\Omega(a) \ll \log d$.
Hence, the number of solutions for 
\[
\sum_{1 \leq j \leq J}r_j \leq \log d
\]
is bound by 
\[
(\log d)^{J} = \exp((\log \log d)^{5/2})
\]
\end{proof}

For any fixed $m \in C_i$ where $i = (a, \mathbf{r}) \in \mathcal{I}$, define the larger auxiliary divisor sum
\[
\mathcal T_i(m):=\sum_{e | m} \sum_{\substack{n \in C_i \\ e | n }} e.
\]
It is enough to bound $\mathcal T_i(m)$ uniformly in $m\in C_i$, because the condition $\gcd(n,m)=e$ in the definition of $T_i$ implies $e|n$, and hence
\[
  dT_i\leq \max_{m\in C_i}\mathcal T_i(m).
\]
Write $m=a \prod_{j=1}^{J}b_j$ and $n=ac=a \prod_{j=1}^{J}c_j$ with $p \mid b_jc_j \Rightarrow p \in \mathcal P_j$
\begin{align}
\sum_{e | m} \sum_{\substack{n \in C(a,\mathbf r) \\ e | n }} e
&=\sum_{e_0 | a} e_0 \sum_{e_1 \mid b_1} e_1\cdots
  \sum_{e_J \mid b_J}e_J \notag\\
&\qquad{}\times
  \#\{c \leq d/a: \Omega_j(c)=r_j,\ e_j \mid c_j \} \notag\\
&=\sum_{e_0 | a} e_0 \sum_{e_1 \mid b_1} e_1\cdots
  \sum_{e_J \mid b_J}e_J \notag\\
&\qquad{}\times
  \#\{q \leq d/(a\prod_{j=1}^{J}e_j):
  \Omega_j(c)=r_j-\Omega(e_1) \}.
\end{align}

So we need some tools from analytic number theory to estimate the counting problem.

Let:
\[
 A(X,\mathbf t)=
 \#\{q\le X:p\mid q\Rightarrow p\in\bigcup_{j\in J}\mathcal P_j,\
       \ \Omega_j(q)=t_j\text{ for every }j\in J\}.
\]
Thus $q$ is allowed to use primes only from the active boxes.  Define
\begin{equation}\label{eq:HjQj}
 H_j=\sum_{p\in\mathcal P_j}\frac1p,
 \qquad
 Q_j=\sum_{p\in\mathcal P_j}\frac1{p^2}.
\end{equation}

\begin{lem}\label{lem:multi-moment}
Suppose that, for every $j\in J$, a real number $z_j$ is chosen with
\begin{equation}
 2\le z_j\le \frac{\mathrm e^{U_j}}2.
\end{equation}
Then
\begin{equation}
 A_I(X,\mathbf t)
 \le X\prod_{j\in J} z_j^{-t_j}
 \exp\!\left(\sum_{j\in J} z_jH_j
              +2\sum_{j\in J}z_j^2Q_j\right).
\end{equation}
\end{lem}

\begin{proof}
By Rankin's trick,
\[
 A_I(X,\mathbf t)\prod_{j\in J}z_j^{t_j}
 \le X \sum_{p \mid q\Rightarrow
        p\in\bigcup_{j\in J}\mathcal P_j}
      \frac{\prod_{j\in J}z_j^{\Omega_j(q)}}{q}
\]
By multiplicativity, 
\[
 A_I(X,\mathbf t)\prod_{j\in J}z_j^{t_j}
 \le X\prod_{j\in J}\prod_{p\in\mathcal P_j}
 \left(1-\frac{z_j}{p}\right)^{-1}.
\]
The claim follows by taking log and the Taylor expansion.
\end{proof}

\begin{proof}[Proof of Proposition~\ref{prop:sieve-theoretic-partition}]
Applying Lemma~\ref{lem:multi-moment}, with
$t_j=r_j-\Omega(e_j)$, gives
\begin{align*}
 \left(\prod_{j=1}^{J}e_j\right)
 \#\{c:e_j\mid c_j\text{ for all }j\}
 &\le \frac da\prod_jz_j^{-(r_j-\Omega(e_j))}\\
 &\qquad{}\times
 \exp\!\left(\sum_jz_jH_j+2\sum_jz_j^2Q_j\right).
\end{align*}
Let $\sigma(a)=\sum_{e_0\mid a} e_0$ denote the sum-of-divisors function. Summing over all
small-part divisors $e_0\mid a$ and all $e_j\mid b_j$, we obtain
\begin{align}
&\sum_{e\mid m}e\,\#\{n\in C:e\mid n\}\notag\\
&\le d\frac{\sigma(a)}a
 \exp\!\left(\sum_jz_jH_j+2\sum_jz_j^2Q_j\right)
 \prod_j
 \sum_{e_j\mid b_j}z_j^{-(r_j-\Omega(e_j))}.
\label{eq:class-master-raw}
\end{align}

We simplify the last divisor sum.  Replace $e_j$ by the complementary divisor
$f_j=b_j/e_j$.  Since $\Omega(b_j)=r_j$,
\begin{align*}
 \sum_{e_j\mid b_j}z_j^{-(r_j-\Omega(e_j))}
 &=\sum_{f_j\mid b_j}z_j^{-\Omega(f_j)}\\
 &=\prod_{p^\alpha\parallel b_j}
   (1+z_j^{-1}+\cdots+z_j^{-\alpha})\\
 &\le (1-z_j^{-1})^{-\omega_j(b)}.
\end{align*}
where:
\begin{equation} \label{eq:omega-j-definition}
  \omega_j(b) = \# \{ p \in \mathcal{P}_j : p | b \}
\end{equation}
The small-prime factor is harmless because the set $p\le\mathrm e^{U_0}$ is fixed:
\[
 \frac{\sigma(a)}a
 \le\prod_{p\le\mathrm e^{U_0}}(1-1/p)^{-1} \ll \log \log \log d.
\]
Finally, for $z\ge2$,
\[
 -\log(1-z^{-1})\le z^{-1}+2z^{-2}.
\]
Taking logarithms in \eqref{eq:class-master-raw} yields the central inequality
\begin{equation}\label{eq:central-inequality}
 \log\frac{\mathcal T_i(m)}d
 \le O(1)
 +\sum_j\left(z_jH_j+\frac{\omega_j(b)}{z_j}\right)
 +2\sum_jz_j^2Q_j
 +2\sum_j\frac{\omega_j(b)}{z_j^2}.
\end{equation}
Choose
\begin{equation}
 z_j=\max\left\{2,\sqrt{\frac{\omega_j(b)}{H_j}}\right\}.
\end{equation}
One can easily verify that $z_j$ is at most $\frac{e^{U_j}}{2}$, as we have insisted before.
Therefore:
\begin{equation}\label{eq:main-reduction}
 \sum_j\left(z_jH_j+\frac{\omega_j(b)}{z_j}\right)
 \le 2\sum_j\sqrt{H_j\omega_j(b)}+4\sum_jH_j.
\end{equation}
By the Mertens' second theorem:
\[
 \sum_jH_j\le\sum_{p\le d}\frac1p \ll \log \log d.
\]
Also:
\begin{equation}
  2 \sum_j \frac{\omega_j(b)}{z_j^2} \leq 2 \sum_j H_j
\end{equation}
and:
\begin{equation}
  \begin{aligned}
    \sum_{j} z_j^2 Q_j
    &\leq 4\sum_j Q_j+\sum_j \frac{\omega_j(b) Q_j}{H_j}\\
    &\ll 1+\sum_{j : U_j \leq 2 \log \log d}
      \frac{|\mathcal{P}_j| H_j}{e^{U_j} H_j}
      + \frac{\log d}{(\log d)^2}\\
    &\ll \sum_{j : U_j \leq 2 \log \log d} e^{\eta U_j}\\
    &\ll (\log \log d)^{5/2}
      \exp(2 (\log \log d)^{1/2})
  \end{aligned}
\end{equation}
due to the fact that $\sum_j \omega_j(b) \leq \log d$ and $Q_j \leq e^{-U_j} H_j$ by the definitions of $Q_j, H_j$ and $\omega_j(b)$ (see \eqref{eq:HjQj} and \eqref{eq:omega-j-definition}).
Thus, the remaining task is to estimate $\sum_j\sqrt{H_j\omega_j(b)}$.

Let:
\begin{equation}
 U_*=\log \log d-4\sqrt{\log \log d}.
\end{equation}

If $U_j<U_*$, then
\[
 H_j\le |\mathcal{P}_j|\mathrm e^{-U_j},
 \qquad \omega_j(b)\le |\mathcal{P}_j|.
\]
Therefore:
\begin{equation}\label{eq:low-onebox}
 \sqrt{H_j\omega_j(b)}\le |\mathcal{P}_j|\mathrm e^{-U_j/2}.
\end{equation}
If $p\in\mathcal P_j$, then
\[
 \log p\le U_{j+1}=(1+\eta)U_j,
\]
so
\[
 \mathrm e^{-U_j/2}\le p^{-1/(2(1+\eta))}.
\]
All primes in the low boxes are at most
\[
 X=\exp((1+\eta)U_*).
\]
Using \eqref{eq:low-onebox} and then enlarging a prime sum to an integer sum,
\begin{align*}
 \sum_{U_j<U_*}\sqrt{H_j\omega_j(b)}
 &\le \sum_{p\le X}p^{-1/(2(1+\eta))}\\
 &\ll \frac{1}{\log X} X^{1-\frac{1}{2(1+\eta)}}\\
 &\ll \frac{1}{\log \log d}
 \exp\!\left(\left(\frac12-\eta-4\eta^2\right)\log\log d\right)\\
 &= \frac{(\log d)^{1/2} \exp(- \eta \log \log d - 4 \eta^2 \log \log d) }{\log \log d} \\
 &\leq \frac{(\log d)^{1/2}}{\log \log d}
\end{align*}
For the remaining boxes, Cauchy--Schwarz gives
\begin{align}
 \sum_{U_j\ge U_*}\sqrt{H_j\omega_j(b)}
 &\le
 \left(\sum_{U_j\ge U_*}U_j\omega_j(b)\right)^{1/2}
 \left(\sum_{U_j\ge U_*}\frac{H_j}{U_j}\right)^{1/2}.
\label{eq:CS-main}
\end{align}
We estimate the two factors separately.

Every distinct prime counted by $\omega_j(b)$ is larger than $\mathrm e^{U_j}$. Hence
\begin{equation}\label{eq:log-budget}
 \sum_jU_j\omega_j(b)
 \le\sum_{p\mid b}\log p
 =\log\operatorname{rad}(b)
 \le\log b
 \le \log d.
\end{equation}
If $p\in\mathcal P_j$, then $\log p\le(1+\eta)U_j$, and therefore
\[
 \frac1{U_j}\le\frac{1+\eta}{\log p}.
\]
By the prime number theorem,
\begin{align}
 \sum_{U_j\ge U_*}\frac{H_j}{U_j}
 &\le(1+\eta)\sum_{p>\mathrm e^{U_*}}\frac1{p\log p}\notag\\
 &=\frac{1+o(1)}{U_*}
 =\frac{1+o(1)}{\log \log d}.
\label{eq:tail-budget}
\end{align}
Combining \eqref{eq:log-budget} and \eqref{eq:tail-budget} in
\eqref{eq:CS-main}, we obtain
\begin{equation}
 \sum_{U_j\ge U_*}\sqrt{H_j\omega_j(b)}
 \le(1+o(1))\sqrt{\frac{\log d}{\log \log d}}.
\end{equation}
Combining this with \eqref{eq:central-inequality}, \eqref{eq:main-reduction}, and the estimates above gives, uniformly in $m\in C_i$,
\[
  \log\frac{\mathcal T_i(m)}d
  \leq (2+o(1))\sqrt{\frac{\log d}{\log\log d}}.
\]
Thus $\max_{m\in C_i}\mathcal T_i(m)\leq d\exp((2+o(1))\sqrt{\log d/\log\log d})$, and the desired bound for $T_i$ follows from $dT_i\leq \max_{m\in C_i}\mathcal T_i(m)$.
\end{proof}

\begin{rmk}
  One can split $[d]$ differently, so that we almost reach desired bounds. For instance let $L = \lfloor \log d \rfloor$ and for each $i \in \mathbb{Z}_+$ let:
  \begin{equation}
    \omega_i(n) = \# \{p \in \mathbb{P} : v_p(n) \geq i \}
  \end{equation}
  Then let
  \begin{equation}
    \mathcal{I} = \{ (k_1, \ldots, k_L) \in \mathbb{Z}_{+}^L : \text{ there exists } n \in [d] : \text{ for all } i \in [L], \omega_i(n) = k_i \}
  \end{equation}
  and
  \begin{equation}
    C_i = \{ n \in [d] : \text{ for all } i \in [L], \omega_i(n) = k_i \}
  \end{equation}
  Then one can use Ramanujan-Hardy result on number of partitions to estimate the size of $|\mathcal{I}|$ and standard bound on partial sum of values of divisor function together with elementary enumerative argument to obtain that:
  \begin{equation}
    T_i \ll \exp((2 + o(1)) \sqrt{\log d \log \log d})
  \end{equation}
  That's said, one can obtain very close estimate without relying on almost any analytic number theory.
\end{rmk}

\section{The key structural lemma for the colored graph}
\label{sec:structural-lemma}

The lemma below gives the dichotomy needed to control the edges between
$K_m$ and $K_n$ whose colors are different from $\gcd(m,n)$. Either an
endpoint lies in one of two small exceptional sets, or a fixed endpoint
outside these sets has weight inversely proportional, up to a factor of
$\log d$, to the number of such edges incident to it.

\begin{lem} \label{lem:relative-structural-decomposition}
  Fix two integers $m, n \in [d]$. Then there exist two subsets $S_m$ and $S_n$ of $K_m$ and $K_n$, respectively, such that the following two sentences are true:
  \begin{enumerate}
    \item $|S_m| \leq \omega(n) + 1$ and $|S_n| \leq \omega(m) + 1$. Here, as usual $\omega(n)$ represents the number of distinct prime divisors of $n$.
    \item For every $v \in K_n \backslash S_n$, denote by $r$ the number of edges from $v$ to $K_m \backslash S_m$ with color different from $\gcd(m, n)$. Then, with the convention that $\frac{\log d}{0}=+\infty$:
    \begin{equation}
      w(v) \ll \min(1, \frac{\log d}{r})
    \end{equation}
  \end{enumerate}
\end{lem}

\begin{proof}
  If $m=n$, set $S_m=S_n=\emptyset$. For two distinct vertices $v,w\in K_n$, the number $\gcd(m(v),m(w))$ is divisible by $n$. If it were a proper multiple of $n$, then either this proper multiple is at most $d$, contradicting the definition of $K_n$, or it is larger than $d$, contradicting \eqref{eq:max-gcd-d}. Hence every such edge has color $n=\gcd(m,n)$, so $r=0$ for every $v\in K_n$ and the claim follows. We may therefore assume that $m\neq n$.

  Let us start by constructing the sets $S_m$ and $S_n$. The definitions for these sets are as follows:
  \begin{equation}
    S_n = (K_n \cap K_m) \cup \{ v \in K_n : \exists_{\substack{p \in \PP \\ v_p(m) > v_p(n)}} v_p(m(v)) > v_p(n) \}
  \end{equation}
  and analogously:
  \begin{equation}
    S_m = (K_n \cap K_m) \cup \{ v \in K_m : \exists_{\substack{p \in \PP \\ v_p(n) > v_p(m)}} v_p(m(v)) > v_p(m) \}
  \end{equation}
  We first check that $K_n\cap K_m$ has at most one element. Indeed, if $v\in K_n\cap K_m$, then $\operatorname{lcm}(n,m)\mid m(v)$. Since $m\neq n$, if $\operatorname{lcm}(n,m)\le d$, then this least common multiple is a proper multiple of at least one of $n$ or $m$, contradicting the corresponding membership in $K_n$ or $K_m$. Hence $\operatorname{lcm}(n,m)>d$. Two distinct vertices in $K_n\cap K_m$ would then have moduli with gcd divisible by $\operatorname{lcm}(n,m)>d$, contradicting \eqref{eq:max-gcd-d}.

  Now take $v\in S_n\setminus (K_n\cap K_m)$ and choose a prime $p=p(v)$ witnessing its membership in $S_n$, so $v_p(m(v))>v_p(n)$. By the restriction on this witness in the definition of $S_n$, we have $p\mid m$. We claim that the map $v\mapsto p(v)$ can be chosen injectively on $S_n\setminus (K_n\cap K_m)$. Indeed, if two distinct vertices $v_1, v_2 \in S_n\setminus (K_n\cap K_m)$ had the same chosen prime $p=p(v_1)=p(v_2)$, then $\min( v_p(m(v_1)), v_p(m(v_2)) > v_p(n)$, and therefore
  \begin{equation}
    pn | \gcd(m(v_1), m(v_2)) 
  \end{equation}
  If $pn\le d$, then $v_1,v_2\in L_{pn}$, contradicting the definition of $K_n$ because $pn$ is a proper multiple of $n$. If $pn>d$, then \eqref{eq:max-gcd-d} is contradicted. Thus the chosen primes are distinct and all divide $m$. This gives the desired bound on the size of $S_n$: $|S_n| \leq \omega(m) + 1$. By symmetry, we have also that $|S_m| \leq \omega(n) + 1$.

  Now fix $v\in K_n\setminus S_n$. Let $w_1,w_2\in K_m\setminus S_m$ be two distinct vertices, and suppose that the edge between $v$ and $w_1$ is colored with $e_1\gcd(n,m)$ and the edge between $v$ and $w_2$ is colored with $e_2\gcd(n,m)$. Then we want to check that:
  \begin{enumerate}
    \item $(e_i, \frac{m}{\gcd(n,m)}) = (e_i, \frac{n}{\gcd(n,m)}) = 1$ for each $i \in \{1, 2\}$.
    \item $(e_1, e_2) = 1$.
  \end{enumerate}
  To prove the first fact, fix $i\in\{1,2\}$. We first show that $(e_i,\frac{n}{\gcd(n,m)})=1$. Suppose, for contradiction, that a prime $p$ divides both $e_i$ and $\frac{n}{\gcd(n,m)}$. Since the edge between $v$ and $w_i$ has color $e_i\gcd(n,m)$, we have $e_i\gcd(n,m)=\gcd(m(v),m(w_i))$, and hence $e_i\gcd(n,m)\mid m(w_i)$. Thus
  \[
    v_p(m(w_i))\geq v_p(e_i\gcd(n,m))>v_p(\gcd(n,m)).
  \]
  On the other hand, $p\mid \frac{n}{\gcd(n,m)}$ implies $v_p(n)>v_p(m)$, so $v_p(\gcd(n,m))=v_p(m)$. Therefore $v_p(m(w_i))>v_p(m)$ and $v_p(n)>v_p(m)$, so $w_i\in S_m$, which gives a contradiction. This proves $(e_i,\frac{n}{\gcd(n,m)})=1$. Similarly, one can show $(e_i,\frac{m}{\gcd(n,m)})=1$.

  Now we prove the second statement. If there is a prime $p$ that divides $e_1$ and $e_2$, then: 
  \begin{equation} \label{eq:e_1-e_2-not-coprime}
    p \gcd(m,n) | \gcd(m(w_1), m(w_2)).
  \end{equation}
  Since we have already proved the first statement, and since $p\mid e_1$, the prime $p$ is coprime to $\frac{m}{\gcd(m, n)}$ and $\frac{n}{\gcd(n, m)}$. Hence $v_p(m) = v_p(n) = v_p(\gcd(n, m))$. Combining \eqref{eq:e_1-e_2-not-coprime} together with $m | \gcd(m(w_1), m(w_2))$, one concludes that:
  \begin{equation}
    pm | \gcd(m(w_1), m(w_2))
  \end{equation}
  Since $pm \leq d$ (by \eqref{eq:max-gcd-d}), one concludes that $w_1 \in L_{pm}$ which is disjoint from $K_m$, so we get a contradiction. Finally, we conclude that $(e_1, e_2) = 1$.

  Let $w_1, w_2, \ldots, w_r$ be all vertices in $K_m \backslash S_m$ so that the edge between $v$ and $w_i$ has color different from $\gcd(n, m)$. Without loss of generality, let $\gcd(m(v), m(w_i)) = e_i\gcd(n,m)$ where $(e_i)_{i = 1}^r$ is an increasing sequence of pairwise coprime numbers. Define recursively a sequence $0=s_0<s_1<\cdots$ as follows. Having chosen $s_j<r$, choose $s_{j+1}$ to be the largest integer with $s_j<s_{j+1}\le r$ such that:
  \begin{equation}
    \prod_{i = 1}^{s_{j+1} - s_j} e_{s_j + i} \leq \frac{d}{\gcd(n, m)}
  \end{equation}
  If $s_{j+1}<r$, repeat the same rule starting from $s_{j+1}$. Since the sequence is strictly increasing and bounded by $r$, the procedure stops after finitely many steps; write the final sequence as
  \[
    0=s_0<s_1<\cdots<s_u=r.
  \]
  For each $j \in [u]$, let:
  \begin{equation}
    E_j = \prod_{i = 1}^{s_j - s_{j-1}} e_{s_{j-1} + i}
  \end{equation}
  Then we get that:
  \begin{equation}
    v \in L_{E_j \gcd(n,m)}
  \end{equation}
  for each $j \in [u]$. For each such $j$, choose $b_j$ maximal such that $b_jE_j \gcd(n,m)\le d$ and $v\in L_{b_jE_j\gcd(n, m)}$. Then $v\in K_{b_jE_j\gcd(n,m)}$. We claim that the $u$ indices $b_jE_j\gcd(n,m)$ are all distinct. Indeed, suppose for some $j<k$ that
  \begin{equation}
    b_j E_j \gcd(n, m) = b_k E_k \gcd(n, m)
  \end{equation}
  with $b_j\le \frac{d}{E_j\gcd(n,m)}$ and $b_k\le \frac{d}{E_k\gcd(n,m)}$. After cancelling $n$, the common value $b_jE_j=b_kE_k$ is a common multiple of $E_j$ and $E_k$. Since the $e_i$ are pairwise coprime, so are $E_j$ and $E_k$, and hence $E_jE_k\mid b_jE_j$. But the maximality in the construction of $s_j$ gives $E_j e_{s_j+1}>\frac{d}{\gcd(n,m)}$, while $E_k\ge e_{s_j+1}$ because $k>j$ and the $e_i$ are increasing. Thus $E_jE_k>\frac{d}{\gcd(n,m)}$, contradicting $b_jE_j\le \frac{d}{\gcd(n,m)}$. Therefore these $u$ indices are distinct, and $v$ belongs to at least $u$ different sets $K_\ell$. It follows that $w(v)\le \frac{1}{u}$. On the other hand, every $e_i$ is larger than $1$, because the corresponding edge color is different from $\gcd(n,m)$. Hence, for each $j \in [u]$,
  \[
    2^{s_j-s_{j-1}}\le E_j\le \frac{d}{\gcd(n,m)}\le d,
  \]
  and so $s_j-s_{j-1}\ll \log d$. Therefore
  \begin{equation}
    r = \sum_{j = 1}^u (s_j - s_{j-1}) \ll u \log d 
  \end{equation}
  giving us the desired $w(v) \ll \frac{\log d}{r}$. Since $w(v) \leq 1$ follows from the definition of the weight, we conclude that the statement of the lemma is true.
\end{proof}

\section{Fourier positivity and the weighted estimate}
\label{sec:weighted-estimate}
We now prove Proposition~\ref{prop:weighted-partition-bound} by estimating a
weighted gcd sum in two ways. M\"obius inversion and Fourier orthogonality give
the lower bound, while pairwise disjointness and
Lemma~\ref{lem:relative-structural-decomposition} give the upper bound. We
begin with the positivity lemma underlying the lower bound.
\begin{lem} \label{lem:mobius-version-of-am-qm}
  For each integer $s \geq 1$ and each sequence of real numbers $(x_i)_{i = 0}^{s-1}$, define, for every divisor $e$ of $s$ and $r = s/e$,
  \[
    a_{r,j}=\sum_{\substack{0\le k<s\\ k\equiv j \pmod r}}x_k\qquad (0\le j<r).
  \]
  Then:
  \begin{equation} \label{eq:mobius-version-of-am-qm}
    \sum_{e\mid s}\mu(e)r\sum_{j=0}^{r-1} a_{r,j}^2\ge 0.
  \end{equation}
\end{lem}

\begin{proof}
  Define the discrete Fourier transform by
  \[
    \widehat x(t)=\sum_{u=0}^{s-1}x_u e^{-\frac{2\pi i t u}{s}}\qquad (0\le t<s).
  \]
  Fix $e\mid s$ and let $r=s/e$. By the orthogonality of additive characters,
  \[
    \mathbf 1_{k\equiv j \pmod r}
    =\frac1r\sum_{\ell=0}^{r-1}e^{\frac{2\pi i\ell(j-k)}{r}}.
  \]
  Thus, for $0\le j<r$,
  \[
    a_{r,j}
    =\frac1r\sum_{\ell=0}^{r-1}e^{\frac{2\pi i\ell j}{r}}
    \sum_{k=0}^{s-1}x_k e^{-\frac{2\pi i\ell k}{r}}
    =\frac1r\sum_{\ell=0}^{r-1}\widehat x(\ell e)e^{\frac{2\pi i\ell j}{r}}.
  \]
  By Parseval's identity,
  \[
    r\sum_{j=0}^{r-1}a_{r,j}^2
    = \sum_{\ell=0}^{r-1}|\widehat x(\ell e)|^2.
  \]
  Hence the left-hand side of \eqref{eq:mobius-version-of-am-qm} is equal to
  \[
    \sum_{e\mid s}\mu(e)\sum_{\ell=0}^{s/e-1}|\widehat x(\ell e)|^2
    = \sum_{t=0}^{s-1}|\widehat x(t)|^2\sum_{e\mid \gcd(s,t)}\mu(e).
  \]
  The inner sum is $1$ if $\gcd(s,t)=1$ and $0$ otherwise. Therefore
  \[
    \sum_{e\mid s}\mu(e)r\sum_{j=0}^{r-1} a_{r,j}^2
    =\sum_{\substack{0\le t<s\\ \gcd(t,s)=1}}|\widehat x(t)|^2\ge 0.
  \]
\end{proof}

\begin{proof}[Proof of Proposition~\ref{prop:weighted-partition-bound}]
  Fix $i\in\mathcal I$.
  Let us consider the following expression:
  \begin{equation} \label{eq:1-epsilon-auxiliary-expression}
    \sum_{n, m \in C_i} \sum_{v_1, v_2 \in V} w(v_1, n) w(v_2, m) \gcd(n, m) \mathbf{1}_{\gcd(n,m) | a(v_1) - a(v_2)}
  \end{equation}
  Firstly, we will try to estimate this expression below, using M\"obius inversion, it is equal to:
  \begin{equation} 
    \sum_{e,r \in [d] : er \leq d} \mu(e) r \sum_{\substack{n, m \in C_i \\ er | n,m}} \sum_{v_1, v_2 \in V} w(v_1, n) w(v_2, m) \mathbf{1}_{r | a(v_1) - a(v_2)}
  \end{equation}
  This can also be expressed as:
  \begin{equation} \label{eq:1-epsilon-auxiliary-expression-2}
    \sum_{q\le d}\sum_{er=q}\mu(e)r\sum_{a=0}^{r-1}\Biggl( \sum_{\substack{n \in C_i \\ q | n}} \sum_{v_1 \in V} w(v_1, n) \mathbf{1}_{a(v_1) \equiv a \pmod r} \Biggr)^2
  \end{equation}
  For each $q\le d$, define
  \[
    c_{u,q}=\sum_{\substack{n \in C_i \\ q | n}} \sum_{v_1 \in V} w(v_1, n)\mathbf{1}_{a(v_1)\equiv u \pmod q}
    \qquad (0\le u<q).
  \]
  If $q=er$ and $0\le a<r$, then the residue classes modulo $q$ lying in the class $a \pmod r$ are exactly the arithmetic progression
  \[
    a,\ a+r,\ \ldots,\ a+(e-1)r.
  \]
  Hence
  \[
    \sum_{\substack{n \in C_i \\ er | n}} \sum_{v_1 \in V} w(v_1,n)\mathbf{1}_{a(v_1)\equiv a \pmod r}
    =\sum_{\substack{0\le u<q\\ u\equiv a \pmod r}} c_{u,q}.
  \]
  Thus the contribution of a fixed $q$ in \eqref{eq:1-epsilon-auxiliary-expression-2} is
  \[
    \sum_{er=q}\mu(e)r\sum_{a=0}^{r-1}\Biggl(\sum_{\substack{0\le u<q\\ u\equiv a \pmod r}} c_{u,q}\Biggr)^2.
  \]
  Applying Lemma \ref{lem:mobius-version-of-am-qm} with $s=q$ and $x_u=c_{u,q}$, this fixed-$q$ contribution is nonnegative for every $q\le d$. Thus each $q$-summand in \eqref{eq:1-epsilon-auxiliary-expression-2} is nonnegative. The contribution of $q=1$ is
  \[
    \Biggl( \sum_{n \in C_i} \sum_{v_1 \in V} w(v_1,n) \Biggr)^2=k_i^2.
  \]
  Therefore \eqref{eq:1-epsilon-auxiliary-expression-2} is bounded below by
  $k_i^2$.  Before using disjointness, separate the diagonal terms with
  $v_1=v_2$. Their contribution
  to \eqref{eq:1-epsilon-auxiliary-expression} is at most
  \[
    \sum_{n,m\in C_i}\sum_{v\in V}w(v,n)w(v,m)\gcd(n,m)
    \leq dT_i\sum_{n\in C_i}\sum_{v\in V}w(v,n)=dT_i k_i,
  \]
  where we used $w(v,m)\leq 1$ and the definition of $T_i$. This is acceptable for the final $O(k_idT_i\log d)$ bound. For the remaining off-diagonal terms, if $\gcd(n,m)=\gcd(m(v_1),m(v_2))$, then the indicator gives the Chinese-remainder compatibility condition, contradicting the pairwise disjointness of $(a_i \pmod{m_i})_{i=1}^k$. Dropping the condition $v_1\neq v_2$ in the resulting upper bound, the off-diagonal contribution is bounded by
  \begin{equation}
    \sum_{n, m \in C_i} \sum_{v_1, v_2 \in V} w(v_1, n) w(v_2, m) \gcd(n, m) \mathbf{1}_{\gcd(n,m) \neq \gcd(m(v_1), m(v_2))}
  \end{equation}
  Let us split the sum above according to the structural behaviour of $K_n$ and $K_m$. First, we estimate the contribution to the above sum from the terms with $v_1\in S_n$ and $v_2\in K_m$:
  \begin{equation} \label{eq:1-epsilon-auxiliary-expression-exceptional}
    \sum_{n, m \in C_i} \sum_{\substack{v_1 \in S_n \\ v_2 \in K_m}}w(v_1, n) w(v_2, m) \gcd(n, m) \mathbf{1}_{\gcd(n,m) \neq \gcd(m(v_1), m(v_2))}
  \end{equation}
  which is bounded by:
  \begin{equation}
    \log d \sum_{n, m \in C_i} \sum_{v_2 \in K_m} w(v_2, m) \gcd(n, m)
  \end{equation}
  This expression can be rearranged as:
  \begin{equation}
    \log d \sum_{m \in C_i} \Bigl( \sum_{v_2 \in K_m} w(v_2, m) \Bigr) \Bigl( \sum_{n \in C_i}  \gcd(n, m) \Bigr)
  \end{equation}
  By the definition of $T_i$, for every fixed $m\in C_i$ we have
  \begin{equation}
    \sum_{n \in C_i} \gcd(n, m) = \sum_{e | m} \sum_{\substack{n \in C_i \\ \gcd(n, m) = e}} e \leq dT_i
  \end{equation}
  Recalling the definition of $k_i$ in \eqref{eq:ki-definition},
  \[
    \sum_{m \in C_i}\sum_{v_2\in K_m}w(v_2,m)\leq \sum_{m \in C_i}\sum_{v_2\in V}w(v_2,m)=k_i,
  \]
  we get \eqref{eq:1-epsilon-auxiliary-expression-exceptional} $\ll k_idT_i\log d$. By symmetry, the same bound holds for the terms with $v_2\in S_m$, so we can focus now on estimating:
  \begin{equation}
    \sum_{n, m \in C_i} \sum_{\substack{v_1 \in K_n \backslash S_n \\ v_2 \in K_m \backslash S_m}}w(v_1, n) w(v_2, m) \gcd(n, m) \mathbf{1}_{\gcd(n,m) \neq \gcd(m(v_1), m(v_2))}
  \end{equation}
  By the arithmetic-geometric mean inequality, we have:
  \begin{equation}
    w(v_1, n) w(v_2, m) \leq \frac{w(v_1, n)^2 + w(v_2, m)^2}{2}
  \end{equation}
  therefore, it is enough to deal with:
  \begin{equation} \label{eq:1-epsilon-auxiliary-expression-standard-2} 
    \sum_{n, m \in C_i} \sum_{\substack{v_1 \in K_n \backslash S_n \\ v_2 \in K_m \backslash S_m}} w(v_1, n)^2 \gcd(n, m) \mathbf{1}_{\gcd(n,m) \neq \gcd(m(v_1), m(v_2))}
  \end{equation}
  Fix $n \in C_i$ and perform the following division of $C_i$:
  \begin{equation}
    C_i = \bigsqcup_{e | n} U(e)
  \end{equation}
  where $U(e) := \{ m \in C_i : \gcd(m, n) = e \}$. For every divisor $e$ of $n$, consider firstly the simpler sum:
  \begin{equation}
    e \sum_{m \in U(e)} \sum_{\substack{v_1 \in K_n \backslash S_n \\ v_2 \in K_m \backslash S_m}} w(v_1, n)^2 \mathbf{1}_{e \neq \gcd(m(v_1), m(v_2))}
  \end{equation}
  By the second claim of Lemma \ref{lem:relative-structural-decomposition}, the number of vertices $v_2$ for which the summand is nonzero is $\ll \frac{\log d}{w(v_1, n)}$, since the condition in the indicator says exactly that the edge color is different from $\gcd(n,m)=e$. Therefore, the above sum is
  \begin{equation}
    \ll e \log d \sum_{m \in U(e)} \sum_{v_1 \in K_n \backslash S_n} w(v_1, n).
  \end{equation}
  Summing this over all divisors $e$ of $n$ and then using the definition of $T_i$, for this fixed $n$ we obtain
  \begin{align}
    \notag
    &\ll \log d \sum_{e|n}e|U(e)|\sum_{v_1 \in K_n \backslash S_n}w(v_1,n)\\
    \notag
    &= \log d \Bigl(\sum_{m\in C_i}\gcd(n,m)\Bigr)\sum_{v_1 \in K_n \backslash S_n}w(v_1,n)\\
    &\leq dT_i\log d \sum_{v_1 \in K_n \backslash S_n}w(v_1,n).
  \end{align}
  Thus the upper bound for \eqref{eq:1-epsilon-auxiliary-expression-standard-2} is:
  \begin{equation}
     \ll T_id \log d  \sum_{n \in C_i} \sum_{v_1 \in K_n \backslash S_n} w(v_1, n) \leq k_i T_i d \log d
  \end{equation}
  Taking together all these observations, we receive the following:
  \begin{equation}
    k_i^2 \ll k_idT_i \log d
  \end{equation}
  Hence, for every $i\in\mathcal I$,
  \[
    k_i\ll dT_i\log d.
  \]
\end{proof}

\end{document}